\documentclass[11pt]{amsart}

\usepackage[colorlinks=true, pdfstartview=FitV, linkcolor=blue, 
      citecolor=blue]{hyperref}
\usepackage{amsfonts, amssymb, amsmath, mdframed, graphicx}
\usepackage{bbm, a4wide}

\DeclareMathAlphabet{\mymathbb}{U}{bbold}{m}{n}

\newcommand{\RR}{\mathbb{R}}
\newcommand{\ZZ}{\ts\mathbb{Z}}
\newcommand{\QQ}{\mathbb{Q}}
\newcommand{\CC}{\mathbb{C}}  

\newcommand{\cF}{\mathcal{F}}
\newcommand{\cO}{\mathcal{O}}

\newcommand{\Rank}{\mathrm{Rank}}
\newcommand{\Ret}{\mathrm{Ret}}

\newcommand{\ts}{\hspace{0.5pt}}

\newtheorem{theorem}{Theorem}
\newtheorem{cor}[theorem]{Corollary}

\theoremstyle{definition}

\newtheorem{example}[theorem]{Example}

\title{How big is a tiling's return module?}

\author{Abigail Perryman}
\address{Department of Mathematics, University of Texas, \newline
  \indent 2515 Speedway, PMA 8.100 Austin, TX 78712, USA}
\email{abbyperryman@utexas.edu}

\author{Lorenzo Sadun}
\address{Department of Mathematics, Univeristy of Texas, \newline
  \indent 2515 Speedway, PMA 8.100 Austin, TX 78712, USA}
\email{sadun@math.utexas.edu}

\begin{document}

\begin{abstract}
The rank of a tiling's return module depends on the geometry of its tiles and is not a topological invariant. However, the rank of the
first \v Cech cohomology $\check H^1(\Omega)$ gives upper and lower bounds for the size of the 
return module. For all sufficiently large patches, the rank of the 
return module is at most the same as the rank of the cohomology. 
For a generic choice of tile shapes and an arbitrary reference patch,
the rank of the return module is at least the rank of $\check H^1(\Omega)$.
Therefore, for generic tile shapes and sufficiently large patches,
the rank of the return module is equal to the rank of $\check H^1(\Omega)$. 
\end{abstract}


\maketitle

\section{Introduction and results}\label{sec:intro}

A basic objective in tiling theory is understanding which properties of
a tiling are topological (depending only on the tiling space up to
homeomorphism), which are combinatorial (depending on which tiles 
touch which other tiles), and which are geometric (depending on the
shapes and sizes of the individual tiles). In this paper we relate 
a geometric object, namely the rank of the return module for
large patches $P$, to a topological object, namely the rank of the 
first \v Cech cohomology of the associated tiling space. 

Let $T$ be a tiling satisfying some basic axioms (see Section \ref{sec:defs}) and let $P$ be a patch. That is, $P$ is 
a finite set of tiles in $T$ whose relative positions are fixed. The
patch $P$ will appear in many different places in $T$. The
position of one occurrence of $P$ relative to another is called a 
{\em return vector} of $P$. The additive group generated by the return
vectors of $P$ is called the {\em return module} of $P$ and is 
denoted $\Ret(P)$. The {\em rank} of $\Ret(P)$ is the maximal number 
of return vectors that are linearly independent over the integers. 
Equivalently, it is the dimension of $\Ret(P) \otimes \QQ$ as a vector
space over the rational numbers $\QQ$. 

If $P'$ is a patch that contains $P$, then every return vector of 
$P'$ is also a return vector of $P$, so the rank of $\Ret(P')$ is 
at most equal to the rank of $P$. For any given tiling, there is a 
limiting rank that applies to the return modules of all sufficiently large patches. 

If $T$ is any tiling, then the {\em tiling space} $\Omega_T$ 
associated with $T$, also called the {\em continuous hull} of $T$, 
is the closure of the orbit of $T$ under translations in the 
``big ball'' metric
where two tilings are considered close if they agree on a big ball
around the origin up to a small translation. There are many
cohomology theories associated with $T$ and $\Omega_T$. For our 
purposes, the most useful is the {\em pattern equivariant} cohomology
\cite{Kel,KP} of $T$, which is isomorphic to the \v Cech cohomology 
$\check H^*(\Omega_T)$. The rank of $\check H^1(\Omega_T)$ is the 
dimension of $\check H^1(\Omega_T, \QQ) = \check H^1(\Omega) \otimes \QQ$
as a vector space over $\QQ$. 

The first of our two main theorems bounds the rank of the return
module of a large patch by the rank of the cohomology.

\begin{theorem}\label{thm:main1}
Let $T$ be an aperiodic tiling that is repetitive and has finite local complexity.
Then for any sufficiently large patch $P$, 
\[ \Rank(\Ret(P)) \le \Rank (\check H^1(\Omega_T)). \]
\end{theorem}

Note that this theorem only bounds the rank of $\Ret(P)$ for 
{\em large} patches. This is what matters, as the return module for 
large patches determines many of the dynamical properties of a tiling,
such as its diffraction spectrum. The return module for small patches
does not have any dynamical significance and can often be changed by 
adding local markings to a tiling. See Section \ref{sec:examples}
for examples of this phenomenon. 

Before stating our second main theorem, we must introduce the concept
of shape changes for tilings. It may happen that two tilings $T$ and 
$T'$ have identical combinatorics, with each tile in $T$ having 
a corresponding tile in $T'$ and with two tiles in $T$ touching if
and only if the corresponding tiles in $T'$ touch. However, the
shapes and sizes of the tiles in $T'$ may be different from those
in $T$. In that case, we say that $T'$ is obtained by applying a 
shape change to $T$. 

For instance, suppose that $T$ is a Thue-Morse 
tiling of the line by two
kinds of tiles, called $a$ and $b$, following the sequence 
$\ldots abbabaabbaababba \ldots$. Suppose furthermore 
that the $a$ and $b$ tiles both have length 1. Now let $T'$ be a tiling
by tiles $A$ and $B$ following the analogous 
pattern $\ldots ABBABAABBAABABBA 
\ldots$, only with an $A$ tile having length $\pi/4$ and a $B$ tile 
having length $\sqrt{2}$. The tilings $T$ and $T'$ have identical 
combinatorics but different geometry. Consequently, the spaces $\Omega_T$ and 
$\Omega_{T'}$ are homeomorphic, but differ as dynamical systems. The return modules of a patch $P \subset T$ and a corresponding patch
$P' \subset T'$ do not necessarily have the same rank. 

The shape changes to a tiling $T$ of $\RR^d$, modulo a form of equivalence called
mutual local derivability (MLD), are parametrized by an open subset of
$\check H^1(\Omega_T) \otimes \RR^d$ \cite{CS1}. 
We say that a property of tilings in this family is {\em generic} if
it occurs for all shapes except for a set of measure zero. (The concept of 
measure zero is clear when $\check H^1$ has finite rank, as we 
are then dealing with a Euclidean space. If $\check H^1$ has infinite rank,
then the set of possible shape changes is a union of finite
dimensional spaces. In that case, we say that a property is 
generic if it applies on a set of full measure on each of these spaces.)

Our second main theorem says that generic shape changes result in
the largest possible return modules. 

\begin{theorem}\label{thm:main2} Let $T$ be an aperiodic tiling that is 
repetitive and has finite local complexity and suppose that $\ell = 
\Rank(\check H^1(\Omega_T))$. Then, after applying a generic shape
change, the rank of every return module is at least $\ell$. 
\end{theorem}

Combining the two theorems, we obtain 

\begin{cor}\label{cor:main} Let $T$ be an aperiodic tiling that is 
repetitive and has finite local complexity and suppose that $\ell = 
\Rank(\check H^1(\Omega_T))$. Then, after applying a generic shape
change, the rank of the return module of every 
sufficiently large patch is exactly $\ell$.
\end{cor}

For simplicity, we have written Theorem \ref{thm:main2} and Corollary \ref{cor:main} assuming that 
$\check H^1(\Omega_T)$ has finite rank. When $\check H^1$ has infinite rank, 
we can restrict our attention to an $\ell$-dimensional subspace 
$H^1_{\ell} \subset \check H^1$. Our proof of Theorem \ref{thm:main2} actually shows that, 
for a generic shape change with 
parameters in $H^1_\ell \otimes \RR^d$,
the rank of the return module of every patch will be at least $\ell$.
Since we can choose $\ell$ to be as large as we wish, there is no
upper bound to the ranks of return modules, even for arbitrarily large patches. 

The dynamical spectrum of a tiling space is closely associated with return modules of 
large patches. Suppose that a large patch $P$ occurs at two locations $x$ and $y$. 
Then $T-x$ and $T-y$ agree on a large ball around the origin and thus are  
close in the tiling metric (see Section \ref{sec:defs} for precise definitions). 
If $f$ is a continuous function on $\Omega_T$, then
$f(T-y) \approx f(T-y)$. In particular, if $f$ is a continuous eigenfunction of translation 
with eigenvalue $\lambda$, then 
\[ \exp(2 \pi i \lambda \cdot (x-y)) \approx 1, \]
with the approximation getting better and better as $P$ gets bigger and bigger. In most tilings of 
interest (in particular, in all primitive substitution tilings), all measurable eigenfunctions can 
be chosen continuous, so this constraint applies to all eigenvalues of translation. 

Solomyak \cite{Sol97} used this observation to relate the point spectrum of a substitution tiling 
to the set of return vectors and to the stretching factor. Baake and Moody \cite{BM} went further, 
using the return module of a tiling with pure point spectrum to reconstruct its 
cut-and-project structure.
For a recent generalization of the Baake-Moody construction, see \cite{Nicu}. 

\section{Definitions and notation}\label{sec:defs}

In this section we review the basics of tilings, tiling spaces and return modules. For a more comprehensive review, see \cite{TAO} and \cite{tilingsbook}.

A {\em tile} is a (closed) 
topological ball that is the closure of its interior. In addition to its geometry, a tile may carry
a label to distinguish it from other tiles of the same size and shape. If two tiles carry the same label,
then each must be a translate of the other. A {\em tiling} of $\RR^d$ is 
a collection of tiles whose union is all of $\RR^d$ and whose interiors are disjoint.  A {\em patch}
of a tiling is a finite subset of its tiles. A tiling is said to have {\em finite local complexity}, or 
FLC, if for each $r>0$ there are only finitely many patches, up to translation, of diameter up to $r$.
Equivalently, a tiling has FLC if two conditions are met: 
\begin{enumerate}
\item There are only finitely many tile types, up to translation, and 
\item There are only finitely many 2-tile patches, up to translation. That is, there are only finitely
many ways for one tile to touch another. 
\end{enumerate}

The group $\RR^d$ of translations acts on tiles by shifting their positions but leaving their labels 
unchanged. By extension, $\RR^d$ acts on tilings and on patches by translating all of the 
tiles simultaneously. The action of $x \in \RR^d$ on a tiling $T$ is denoted $T+x$. 
A tiling $T$ is said to be {\em aperiodic} if $T+x = T$ implies $x=0$. The orbit of a tiling $T$ under
translation is denoted $\cO(T)$. 

We will frequently consider a particular pattern of adjacent tiles, 
such as the pattern $P=aaba$ in a Thue-Morse tiling, 
without specifying the location of $P$. Strictly speaking, $P$ is an 
equivalence class
of patches under translation rather than a specific patch, but we will abuse notation and terminology by 
calling $P$ a ``patch'' anyway. 
We can then talk about multiple occurrences of $P$ in a tiling, i.e. multiple patches in the 
equivalence class defined by $P$. 

On the set of all tilings by a given set of tile types, we consider the topology induced by the {\em big 
ball metric}, in which two tilings $T$ and $T'$ are considered $\epsilon$-close if they agree on a 
ball of radius $1/\epsilon$ around the origin, up to translation of each by a distance $\epsilon/2$ or less.
A {\em tiling space} is a non-empty translation-invariant set of tilings that is closed in this topology. 

We can obtain a tiling space from any tiling $T$ by taking the closure of $\cO(T)$. This space is called 
the {\em continuous hull} of $T$ and is denoted $\Omega_T$. It is the smallest tiling space that contains
$T$. A tiling $T'$ is in $\Omega_T$ if and only if every patch of $T'$ is a translate of a patch in $T$. 

A tiling $T$ is said to be {\em repetitive} if, for every patch $P \subset T$ there exists a radius $r(P)$
such that every ball of radius $r(P)$ contains at least one copy of $P$. This is equivalent to $\Omega_T$
being a minimal dynamical system. That is, every orbit is dense, so for each $T' \in \Omega_T$, 
$\Omega_{T'} = \Omega_T$. In that case, any two tilings $T_1,T_2 \in \Omega_T$ have exactly the same patches 
(up to translation), so any quantity based on those patches is the same for both tilings.  

Suppose that $T$ and $T'$ are two tilings with the following property: There is a radius $r$ such 
that, for all $x, y \in \RR^d$ such that $T-x$ and $T-y$ agree on a ball of radius $r$ around the origin,
$T'-x$ and $T'-y$ agree on a ball of radius 1 around the origin. This is a precise way of saying that
the patterns of $T'$ are determined in a local way from the patterns of $T$ in the exact same locations.
If this condition is met, then we say that $T'$ is {\em locally derivable}, or LD, from $T$. If $T'$ is LD from $T$ and $T$ is LD from $T'$, then $T$ and $T'$ are {\em mutually locally derivable}, or MLD. The
local rule deriving $T$ from $T'$ extends to a topological conjugacy from $\Omega_T$ to $\Omega_{T'}$ that 
we call an {\em MLD equivalence}. 

In principle, the tiles in a tiling may have very complicated (say, fractal) boundaries. However, 
every FLC tiling is MLD to a tiling whose tiles are convex polygons (or polytopes) that meet full edge
to full edge. We can therefore assume, with no loss of generality, that our tilings are of this sort. 

The procedure for doing this conversion is called the ``Voronoi trick''. It involves first 
picking a ``control point'' for each tile (or for each instance of a more complicated patch $P$) 
and replacing the tiling with a point pattern, and then replacing
each point in the point pattern with its ``Voronoi cell'' consisting of all points in $\RR^d$ that are 
closer to the given control point than to any other control point. For more information on the 
Voronoi trick, see \cite{TAO} or \cite{tilingsbook}. 

Now pick a patch $P$ and pick a control point within $P$ to represent the patch. (For instance, we might pick the control point of the pattern $abaa$ to be the left 
endpoint of the $b$ tile.) When we speak of the locations of $P$ in a tiling, we mean the locations of 
the control point. Let $\{x_1, x_2, \ldots\}$ be all the locations 
of $P$ in $T$. The relative positions $x_i-x_j$ are called {\em return vectors} for $P$ and do not
depend on the choice of control point. The span
(over $\ZZ$) of the the return vectors is called the {\em return module} of $P$ and is denoted $\Ret(P)$. 
The fact that $x_i-x_j$ is a return vector is a property of the ball of radius slightly larger than 
$|x_i-x_j|/2$ centered at 
the point $(x_i+x_j)/2$. If $T$ is repetitive, then a copy of this large ball appears in 
every $T' \in \Omega_T$, so 
$x_i-x_j$ is also a return vector for $P$ in $T'$. The set of return vectors and the 
resulting return module, are thus quantities that we can associate with the tiling space $\Omega_T$
rather than with just the specific tiling $T$.

If $P$ and $P'$ are patches with $P \subset P'$, then every return vector of $P'$ is also a return vector
of $P$, so $\Ret(P') \subset \Ret(P)$ and $\Rank(\Ret(P')) \le \Rank(\Ret(P))$. We are interested in the limit 
of this rank as the patches grow to $\infty$. 

Note that we are taking the limit of the rank, not the rank of the limit! In some tilings, the return
module for a large patch can be smaller, but of the same rank, than the return module for a small patch. 
For instance, in the Thue-Morse tiling with $a$ and $b$ tiles both having length 1, the return module for
the one-letter patch $P_1=a$ is $\ZZ$ while the return module of $P_1=abb$ is $2\ZZ$, and there are
other patches whose return modules are $4\ZZ$, $8\ZZ$, etc. The limit of these
modules is the rank-0 set $\{0\}$, but the rank of $2^n \ZZ$ is 1 for all $n$, 
so the limit of the rank is 1. (In other tilings, such as the Fibonacci tiling, all patches have the same 
return module and there is no need to take a limit at all.) 

If $T$ is repetitive, then 
the details of how we take the limit of large patches are unimportant. Any patch $P_i$ has a 
repetitivity radius $r(P_i)$. Any patch $P_j$ that contains a ball of radius $r(P_i)$ must therefore 
contain a copy of $P_i$. In order to take a limit over all patches with our partial ordering, it is
sufficient to consider a single sequence $P_1 \subset P_2 \cdots$ of patches such that the inner diameter of the 
$P_i$'s goes to $\infty$. In particular, we could pick the $P_i$'s to come from balls of increasing
radius around the origin in a specific tiling $T$. 

\section{Tiling cohomology}

The precise definition of the \v Cech cohomology of a space is complicated, involving open covers, nerves
of said covers, the simplicial cohomology of those nerves, and a limit over all open covers partially 
ordered by refinement \cite{Hatcher}. 
Fortunately, those details are not needed for a working understanding of tiling 
cohomology. For our purposes, two facts are sufficient \cite{CohoChapter}: 
\begin{enumerate}
\item The \v Cech cohomology of a CW complex is isomorphic to the singular cohomology, which in turn is
isomorphic to many other cohomology theories. On a CW complex, ``all cohomologies are the same''. 
\item The \v Cech cohomology of an inverse limit space is the direct limit of the \v Cech cohomologies
of the approximants. 
\end{enumerate}
Tiling spaces are not CW complexes, but they {\em are} inverse limits of CW complexes. Each approximant, called
an Anderson-Putnam complex, describes the tiling in a neighborhood of the origin. A point in the inverse limit of the approximants is a set of consistent
instructions for tiling bigger and bigger portions of $\RR^d$. The union of the neighborhoods is all of 
$\RR^d$, so a point in the inverse limit is a set of instructions for tiling all of space, which is 
tantamount to a tiling itself. There are numerous ways to construct approximants for tiling spaces, but 
they are all qualitatively similar. The key ideas are due to Anderson and Putnam \cite{AP} 
and to G\"ahler \cite{Gaehler}. 
See \cite{inverse} for a unification of their arguments and \cite{tilingsbook} for a review. 

The upshot is that the \v Cech cohomology of a tiling space is constructed from data about the local 
structure of a tiling, where ``local'' can include information out to any finite distance but not out
to infinity.  This idea was codified by Kellendonk and Putnam \cite{Kel,KP} as a new cohomology theory
called Pattern Equivariant (PE) cohomology. 

Consider a reference tiling $T$ of $\RR^d$ and a function $f: \RR^d \to \RR$. We say that a function $f: 
\RR^d \to \RR$ is PE with
radius $r$ if its value at a point $x$ depends only on the form of $T$ in a ball of radius $r$ around $x$.
More precisely, if $f$ is PE with radius $r$ and if $T-x_1$ and $T-x_2$ agree on a ball of radius $r$ 
around the origin, then $f(x_1)=f(x_2)$. A function is said to be PE (without any qualifiers) if it is PE 
with some finite radius. We can likewise define PE differential forms. It is easy to check that the
exterior derivative of a PE form is PE. Kellendonk and Putnam defined the
(real-valued) PE cohomology of a tiling $T$ to be 
\[ H^k_{PE}(T) = \frac{\hbox{Closed PE $k$-forms}}{d(\hbox{PE $k-1$-forms})} \]
and proved that $H^k_{PE}(T)$ was isomorphic to $\check H^k(\Omega_T) \otimes \RR$. 

To obtain an integer version of PE cohomology, we can consider PE cochains. A PE $k$-cochain with radius $r$ assigns 
an integer to every $k$-cell in the tiling $T$ based on the neighborhood of size $r$ around that $k$-cell. (If $T$ does not have clearly defined vertices, edges, 
faces, etc., apply the Voronoi trick to convert $T$ into a tiling that does.) That is, if the tiling is the same
within some fixed distance $r$ of two different $k$-cells, then our function must assign the same value 
to both cells. A cochain is said to be PE if it is PE with some radius. 

The coboundary of a cochain is defined the same as with ordinary (not PE) cochains. If $\alpha$ is a 
$k$-cochain and $c$ is a $(k+1)$-cell, then $\delta \alpha$ is a $(k+1)$-cochain whose value on $c$ is
\[ \delta \alpha (c) := \alpha(\partial c), \]
where $\partial c$ is the boundary of $c$. The coboundary of a PE cochain is PE (although possibly with
a slightly larger radius). Sadun \cite{integer} defined the integer-valued PE cohomology of $T$ to be 
\[ H^k_{PE}(T) = \frac{\hbox{Closed PE $k$-cochains}}{\delta(\hbox{PE $(k-1)$-cochains})} \]
and proved that this was isomorphic to $\check H^k(\Omega_T)$. 
The same construction works with values in any Abelian group, not just the integers $\ZZ$. 

We are interested in the rank of $\check H^1(\Omega_T)$. This is the same as the dimension over $\QQ$ of 
$\check H^1(\Omega_T) \otimes \QQ$, which is naturally isomorphic to $H^1_{PE}(T,\QQ)$, the first PE 
cohomology of $T$ with values in $\QQ$. All of our calculations will be done in the PE setting, where
we represent cohomology classes with cochains on $T$.  

The following theorem will allow us to restrict our attention to return vectors of large patches. 
\begin{theorem}\label{thm:Just-returns}
Let $P$ be a patch in a repetitive tiling $T$ and suppose that the locations of $P$ are 
$\{x_1, x_2, \ldots\}$. If $\alpha$ is a closed PE 1-cochain, then the cohomology class of $\alpha$
depends only on the value of $\alpha$ applied to paths from $x_i$ to $x_j$. That is, 
\begin{enumerate}
\item If a closed PE 1-cochain $\alpha$ evaluates to zero on a path from $x_i$ to $x_j$ for each pair $(i,j)$, then
$\alpha$ represents the zero cohomology class. 
\item If two closed PE 1-cochains $\alpha$ and $\beta$ give the same values on a path from each $x_i$ to each other
$x_j$, then $\alpha$ and $\beta$ represent the same cohomology class. 
\end{enumerate}
\end{theorem} 

\begin{proof} We begin with the first statement. 
Let $r_1$ be the PE radius of $\alpha$ and let $r_2$
be the recognizability radius of $P$. For each vertex $z$ in our 
tiling, let $f(z) = \alpha(c)$, where $c$ is a path from $x_1$ to $z$. 
(Since $\alpha$ is closed, this does not depend on our choice of 
path.) We manifestly have $\alpha = \delta f$. 

We pick our path $c$ to be the concatenation of two paths $c_1$
and $c_2$, where $c_1$ goes from $x_1$ to an $x_i$ that is close to $z$ and $c_2$ goes 
from $x_i$ to $z$.  See Figure \ref{fig:thm4}. We can also pick the path $c_2$ to stay within
a distance $r_2$ of $z$. Since $\alpha(c_1)=0$, $f(z)=\alpha(c_2)$.
However, $\alpha(c_2)$ only depends on the pattern $T$ within a 
distance $r_1$ of all of the edges in $c_2$, and therefore within
a distance $r_1+r_2$ of $z$. This makes $f$ a PE function with 
radius $r_1+r_2$. Since $\alpha$ is the coboundary of a PE function,
$\alpha$ represents the zero class in cohomology.

\begin{figure}[ht]
    \centering
    \includegraphics[width=0.7\textwidth]{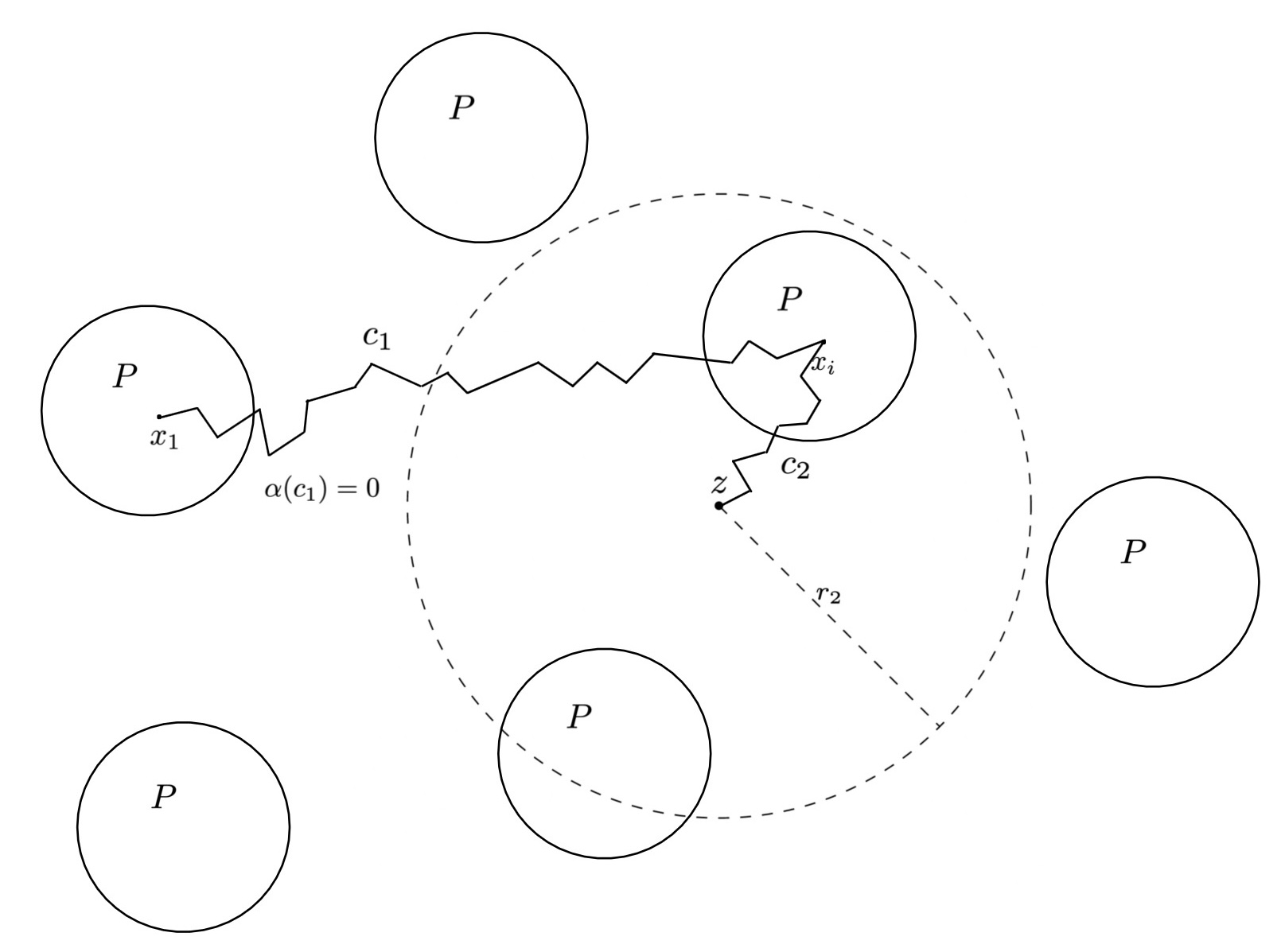}
    \caption{The function $f(z) = \alpha(c_1)+\alpha(c_2) = \alpha(c_2)$ is pattern-equivariant with radius $r_1+r_2$.}
    \label{fig:thm4}
\end{figure}

For the second statement, we simply apply the first statement to 
$\alpha - \beta$. Since $\alpha-\beta$ evaluates to zero on all 
paths from $x_i$ to $x_j$, the class of $\alpha-\beta$ is zero, so 
the class of $\alpha$ is the same as the class of $\beta$. 
\end{proof}

In fact, we can go a step beyond Theorem \ref{thm:Just-returns}. Using
a dual version of the Voronoi trick, we can construct a tiling $T'$, MLD
to the original tiling $T$, whose vertices are the points $x_1, x_2, \ldots$. 
Since $\Omega_T$ and $\Omega_{T'}$ are homeomorphic, we 
can identify the \v Cech cohomology of $\Omega_T$ with the \v Cech
cohomology of $\Omega_{T'}$, which in turn is isomorphic to the PE
cohomology of $T'$. As a result, we can represent classes in 
$\check H^1(\Omega_T)$ with closed PE 1-cochains on $T'$. That is, we can represent classes 
in $\check H^1(\Omega_T)$ 
as functions on edges in $T'$ 
that connect pairs of locations of $P$ in $T$. In particular, every
linear function $L: \Ret(P) \to \ZZ$ defines a cohomology class in $\check H^1(\Omega_T)$
that we denote $\phi(L)$. 

\section{Proofs of main theorems}

\begin{proof}[Proof of Theorem \ref{thm:main1}]
Let $k$ be the minimum rank of the return module of any patch and 
let $P_0$ be a patch for which $\Rank(\Ret(P_0))=k$. Every sufficiently
large patch $P$ (e.g., any patch whose inner radius is greater than
the repetitivity radius of $P_0$) will contain a copy of $P_0$ and 
so will have $\Rank(\Ret(P)) \le \Rank(\Ret(P_0)) =k$. But $k$ is the 
minimum possible rank, so we must have $\Rank(\Ret(P))=k$. 

Let $P$ be any patch whose return module has rank $k$. 
As noted in the comment after the proof of Theorem \ref{thm:Just-returns}, every linear function $L: \Ret(P) \to \ZZ$ 
defines a closed
1-cochain $\alpha_L$ on an associated tiling $T'$ whose vertices are the
locations of $P$ in $T$, and thus defines a cohomology class 
$\phi(L) \in H^1_{PE}(T') \simeq \check H^1(\Omega_T)$. 

We claim that the map $\phi$ from linear functions to cohomology classes is injective. 
To see this, suppose that 
$\phi(L)$ is the zero class in cohomology. Then the 1-cochain $\alpha_L$ 
on $T'$ that is defined by $L$ must be a coboundary: 
\[ \alpha_L = \delta f, \]
where $f$ is a PE function with some radius $r$. Let $P'$ be a 
patch of $T'$ containing a ball of radius $r$ centered at a point $x_i$. Then $\alpha$ applied to any return vector of $P'$ must be
zero, insofar as $f$ takes on the same value at the endpoints of 
a chain connecting two instances of $P'$. Thus, $L$ restricted
to $\Ret(P') \subset \Ret(P)$ is zero. 

However, the rank of $\Ret(P)$ is already the minimum among all patches
of $T$. This implies that $\Ret(P')$ is a submodule of $\Ret(P)$ of full rank, 
so the only linear function on $\Ret(P)$ that vanishes on 
$\Ret(P')$ is the zero function. That is, $\phi(L)=0$ implies
that $L=0$. 

The rank of $\Ret(P)$ is the same as the rank of the space of linear
functions $\Ret(P) \to \ZZ$, which (by the injectivity of $\phi$) 
is the same as 
the rank of the image of $\phi$, which is bounded by the rank of 
$\check H^1(\Omega_T)$. Thus for all sufficiently
large patches $P$, the rank of $\Ret(P)$ is bounded by the rank of 
$\check H^1(\Omega_T)$. 
\end{proof}

Before proving Theorem \ref{thm:main2}, we review the way that the shapes and sizes of all tiles are 
parametrized by cochains. In any tiling $T$, there is a vector-valued cochain $\cF(T)$, called the 
{\em Fundamental Shape Cochain} of $T$, that assigns to every edge the actual displacement along that edge.
This cochain is closed, since the net displacement along the boundary of any 2-face is zero. 
The corresponding class $[\cF(T)] \in H^1_{PE}(T,\RR^d)$ is called the {\em Fundamental Shape Class}. 

To obtain a tiling with the same combinatorics as $T$ but different geometry, we deform the 
cochain $\cF(T)$ into another cochain $S$. That is, we construct a new tiling (denoted $S(T)$) whose 
vertices, edge, faces, etc. are in 1--1 correspondence with those of $T$, carrying the same labels, such
that the relative position of any pair of vertices in $S(T)$ is given by $S$ applied to a path 
connecting the corresponding vertices of $T$. See Figure \ref{fig:shapechanges} for an example.

\begin{figure}[ht]
    \centering
    \includegraphics[width=0.75\textwidth]{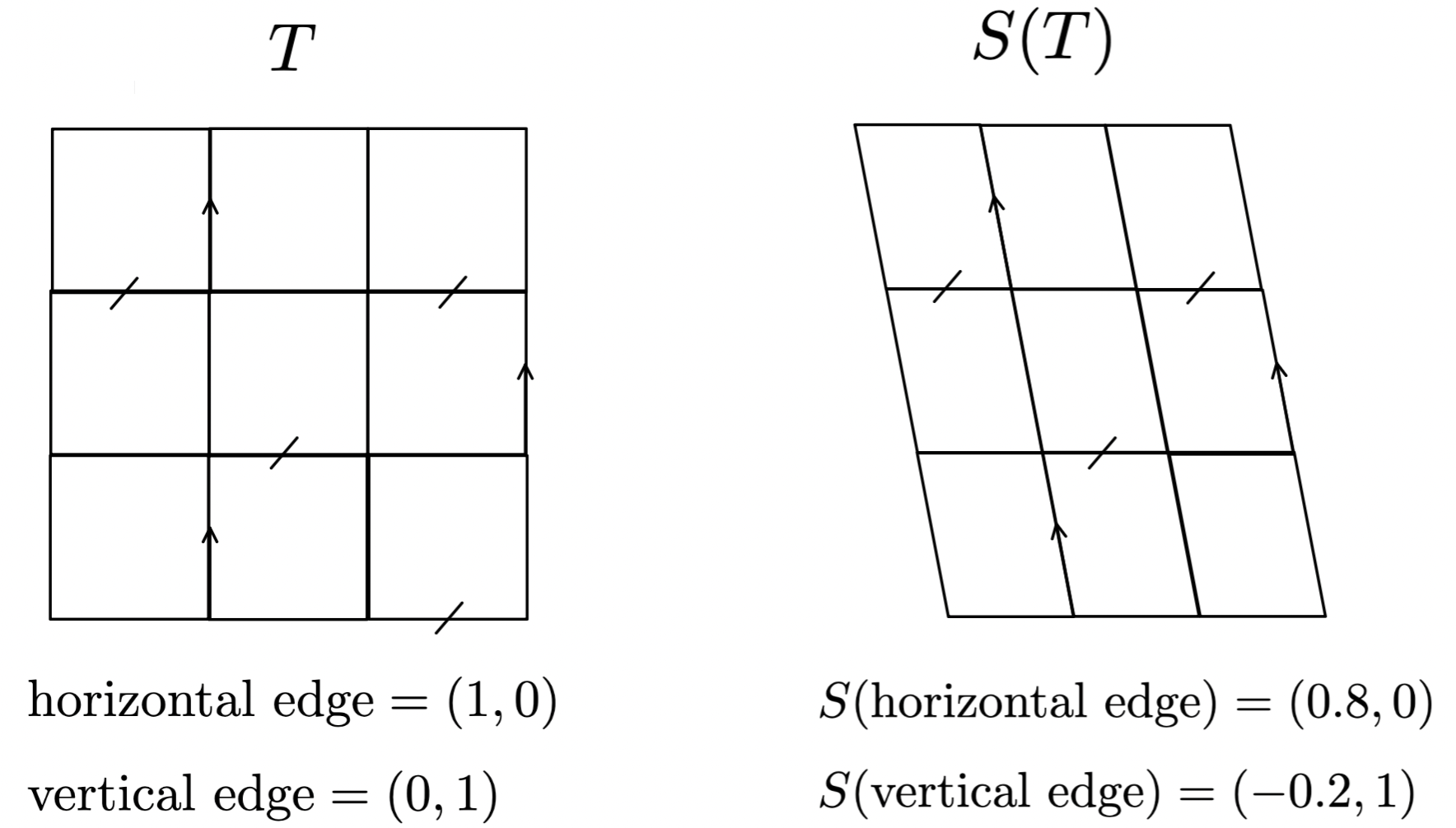}
    \caption{By changing the shape cochain to $S$, we convert a tiling by squares into
    a tiling by parallelograms.}
    \label{fig:shapechanges}
\end{figure}

In order for the tiles of $S(T)$ to fit together and preserve finite local complexity, we need
three conditions:
\begin{itemize}
\item The (vector-valued) cochain $S$ must be PE. 
\item $S$ must be closed, and 
\item The new shapes must not turn any of the tiles inside-out. If $d=2$, this is equivalent to the 
boundary of each tile being a positively oriented simple closed curve. (Figure 8's are not allowed!). 
In higher dimensions the condition is similar, although trickier to describe. 
\end{itemize}
The third condition is open. That is, any closed PE cochain $S$ that is close to $\cF(T)$
will automatically satisfy it. As long as we concentrate on small shape changes, we only
need to keep track of the first two conditions. 

(While not needed for this paper, large shape change classes are indeed possible. Given a cochain $S$ that
turns some tiles inside-out, it is usually possible to find another closed PE cochain $\tilde S$, 
cohomologous to $S$, that satisfies the third requirement. The set of classes in $H^1_{PE}(T,\RR^d)$ 
that can be realized as fundamental shape classes is dense, open, and of full measure. 
See \cite{JS} for details.)  

In general, if $A$ is any object associated with $T$, then we write $S(A)$ to denote the corresponding
object in $S(T)$, and if $B$ is any object in $S(T)$, then we write $S^{-1}(B)$ to denote the 
corresponding object in $T$. Thus the hull of $S(T)$ is $S(\Omega_T) = \Omega_{S(T)}$, and for each 
patch $P$ in $T$ we are interested in the rank of $\Ret(S(P))$. 

Note that $\Omega_T$ and $S(\Omega_T)$ are homeomorphic and so have identical cohomology groups. All 
statements about cohomology apply both to $T$ and to $S(T)$. 

If the rank of $\check H^1(\Omega_T)$ is $\ell$, then $\check H^1(\Omega_T,\RR^d) \simeq \RR^{\ell d}$
is a Euclidean space. On this space, the sets of zero (Lebesgue) measure are clearly defined. The
essence of Theorem \ref{thm:main2} is that, after applying a shape change with the class of $S$ avoiding
a certain set of measure zero, we obtain return modules of rank at least $\ell$. 

\begin{proof}[Proof of Theorem \ref{thm:main2}]

Let $e_1, \ldots, e_d$ be the standard basis for $\RR^d$. Since $\RR^d$ has infinite dimension over $\QQ$, 
we can find $\ell$ vectors $w_1, \ldots, w_\ell \in \RR^d$ such that $(e_1,\ldots,e_d,w_1,\ldots,w_\ell)$
are linearly independent over $\QQ$. That is, there is no way to write 0 as a nontrivial 
linear combination of these vectors with rational coefficients. 
Among all $\ell$-tuples of vectors in $\RR^d$, those that
meet this condition have full measure, since there are only countably many possible rational linear
combinations and each one is zero only on a set of codimension $d$. We can also shrink 
the vectors 
$w_i$ by an integer factor to make them arbitrarily small without affecting linear independence
over $\QQ$. 

Since the rank of $\check H^1(\Omega_T)$ is $\ell$, we can choose $\ell$ closed PE cochains 
$\alpha_1, \ldots, \alpha_\ell$
whose classes form a linearly independent set 
in $H^1_{PE}(T, \QQ) \simeq \check H^1(\Omega_T) \otimes \QQ$. 

We do our shape change in two steps. First we change to a new cochain $S_0$ that evaluates to a 
rational vector on each edge. The set of such rational shape cochains
is dense in the set of all shape cochains \cite{SW}, so we can pick $S_0$ arbitrarily close to our original $\cF(T)$. 
We then define 
\[ S = S_0 + \sum_{i=1}^\ell \alpha_i \otimes w_i. \]
What remains is picking an arbitrary patch $P$ in $T$ 
and showing that $\Rank(\Ret(S(P))) \ge \ell$. 

The displacement along any edge $e$ of $S(T)$, which is the same as $S$ applied to $S^{-1}(e)$,
is a rational vector (namely the contribution of $S_0$) plus a rational linear combination of the $w_i$'s. 
Let $I_i$, applied to a rational linear combination of $(e_1,\ldots,e_d,w_1,\ldots,w_\ell)$, 
give the coefficient of $w_i$. This applies in particular to the return vectors of $S(P)$ and indeed 
to all elements of $\Ret(S(P))$. Note that if $v_1$ and $v_2$ are vertices in $T$, then $I_i$ applied to the displacement $S(v_2)-S(v_1)$ is the same as $c_i$ applied to a 1-chain from $v_1$ to $v_2$. 
The collection $(I_1, \ldots, I_\ell)$ defines
a linear map $\varphi: \Ret(S(P)) \to \QQ^\ell$. We claim that the image of $\varphi$ has rank $\ell$, which
then implies that $\Ret(S(P))$ has rank at least $\ell$, as claimed in Theorem \ref{thm:main2}.

To see that the image of $\varphi$ contains $\ell$ linearly independent elements, suppose otherwise. Then
there exists a nonzero vector $b= (b_1,\ldots,b_\ell) \in \QQ^\ell$ that is orthogonal to 
$\varphi(v)$ for every return vector $v$ of $S(P)$. However, 
\[ 0 = b \cdot v = \sum (b_i \alpha_i) (S^{-1}(v)). \]
This implies that 
$\sum b_i \alpha_i$ evaluates to zero on all return vectors of $P$. By Theorem \ref{thm:Just-returns}, this
then implies that the cohomology class of $\sum b_i \alpha_i$ is zero. However, the cochains $\alpha_i$ were assumed to
represent linearly independent cohomology classes and $b$ was assumed to be nonzero. Contradiction. 

We have obtained a return module of rank $\ell$ for any set of cochains $\alpha_i$ whose 
cohomology classes form a basis for $\check H^1(\Omega, \QQ)$ 
and for almost every set of vectors $w_i$ in a neighborhood
of zero. This was done for a particular starting set of tile shapes, but the same argument works for 
any starting set. As a result, the set of shape classes for which all return modules have rank $\ell$
or greater has full measure in the $\ell d$-dimensional space of possible shape classes. 
\end{proof}

Note that, prior to the last paragraph, the proof never uses the assumption that the dimension of 
$\check H^1(\Omega_T, \QQ)$ was 
$\ell$. It merely uses the existence of $\ell$ linearly independent cohomology classes represented 
by closed PE cochains $c_1, \ldots, c_\ell$. 
If $\check H^1(\Omega_T, \QQ)$ is infinite-dimensional, then we can find such a set of classes (and 
cochains) for any positive integer $\ell$. 
A shape change using generic shape classes in $\check H^1_\ell \otimes \QQ$, where $\check H^1_\ell$ 
is the span of the cohomology classes of the $\alpha_i$'s, will then give us return modules of rank at 
least $\ell$. This justifies the comment after the statement of Corollary \ref{cor:main}. 

\section{Examples} \label{sec:examples}

\subsection{One dimensional examples}

\begin{example}[Fibonacci] The Fibonacci tiling is based on the substitution $a \to ab$, $b \to a$. The first
cohomology of the resulting tiling space has rank 2: $\check H^1(\Omega) = \ZZ^2$. Let $P_0$ be a
one-tile patch consisting of an $a$ tile. The distance from each $P_0$ to the subsequent $P_0$ is either the length of 
an $a$ tile (if the two $P_0$'s are back-to-back, as in $aa$) or the length of an $a$ tile plus the length of a $b$ 
tile (if the two $P_0$'s are separated by a $b$, as in $aba$). 
This means that $\Ret(P_0)$ is the span of the length
of an $a$ tile and the length of a $b$ tile. This has rank 1 if the ratio of the two lengths is rational
and has rank 2 if the ratio is irrational. 

What about bigger patches? Every large patch $P$ must contain a supertile of some order $m$ and must 
be contained in some supertile $M$. This means that we must have 
\[ \Ret(P_2) \subset \Ret(P) \subset \Ret(P_1), \]
where $P_1$ is a supertile of order $m$ and $P_2$ is a supertile of order $M$. However, $\Ret(P_1)$ and 
$\Ret(P_2)$ are the spans of the lengths of the two kinds of supertiles of order $m$ and $M$, respectively. 
Since the substitution matrix $\left ( \begin{smallmatrix} 1 & 1 \cr 1 & 0 \end{smallmatrix} \right )$ 
is invertible over $\ZZ$, having determinant $-1$, both of these are the 
same as the span of the lengths of the $a$ and $b$ tiles. In other words, all patches $P$ have exactly 
the same return module as $P_0$ and in particular have the same rank: 1 if $|a|/|b| \in \QQ$ and 2 
if $|a|/|b| \not \in \QQ$. 

Note how these results fit with the general scheme of Theorems \ref{thm:main1} and \ref{thm:main2}. 
The cohomology has rank 2, so the return modules of large patches cannot have rank greater than 2.
With a generic choice of tile lengths, return modules have rank equal to 2. Specifically, the return
modules of large (and small) patches have rank 2 unless the ratio $|a|/|b|$ lies in a countable union
of codimension-1 subsets of $\RR$, namely the rational numbers $\QQ$. 

So far we have only considered versions of the Fibonacci tiling where all $a$ tiles have length $|a|$ and
all $b$ tiles have length $|b|$. We could also describe the Fibonacci tiling using collared tiles, with
several different varieties of $a$ tile, each with its own length, and several different varieties of 
$b$ tile. The return module of small patches could then have rank greater than 2. For instance, the
return vectors for vertices of arbitrary type are spanned by the lengths of all the different collared 
tiles. By collaring to a big enough radius and picking tile lengths that are rationally independent, we can
get this rank to be as large as we wish. However, the return modules of {\em large} patches can only 
have rank 1 or 2. 

Incidentally, every Fibonacci tiling space obtained by collaring and varying the lengths of the collared
tiles turns out to be mutually locally derivable (MLD) to a Fibonacci tiling space involving uncollared 
tiles. (This is because two shape classes in the same cohomology class give rise to 
MLD tilings and because the generators of $\check H^1(\Omega)$ can be expressed in terms 
of uncollared tiles.) Tilings that are MLD have exactly the same return modules for large patches, but not necessarily
for small patches. 

\end{example}

\begin{example}[Sturmian tilings] Our analysis of the Fibonacci tiling relied on its 
substitutive structure, but it is possible to derive similar results for any Sturmian tiling. 
A Sturmian tiling (see e.g. \cite{Arnoux}) is a canonical cut-and-project tiling from 2 
dimensions to 1. There
are two kinds of tiles, which we call $a$ and $b$, with the ratio of frequencies being an irrational number $\alpha$. Such a tiling can be described via a substitution if and only if $\alpha$ is a quadratic irrational. The Fibonacci tiling is the simplest such case, with $\alpha$ being the golden ratio 
$\phi = (1+\sqrt{5})/2$. 

The first cohomology of any Sturmian tiling space is $\ZZ^2$, with generators that count $a$ 
and $b$ tiles, respectively. As with the Fibonacci tiling, the return module for the smallest possible
patch (a single tile) is generated by the lengths of the two tiles. This has rank 1 if $|a|/|b|$ is
rational and rank 2 if $|a|/|b|$ is irrational. If we collar the Sturmian tiling and vary the lengths of 
the collared tiles, we can get the return modules of small patches to have arbitrarily large rank, but 
the return modules of large patches always have rank 2 for generic choices of the lengths of the $a$ and 
$b$ tiles and rank 1 for countably many values of $|a|/|b|$. 

\end{example}

\begin{example}[Thue-Morse] The Thue-Morse tiling is based on the substitution 
$a \to ab$, $b \to ba$. Note that both kinds of supertiles have one $a$ tile and
one $b$ tile. No matter what lengths we assign to the basic tiles, all supertiles
have the same length $|a|+|b|$. 

Any patch $P$ of length 5 or greater must have two consecutive letters of the same type, as the patterns $ababa$ and $babab$ never appear. However, $aa$ or $bb$ can
only appear when there is a supertile boundary in between the two $a$'s or two $b$'s. 
As a result, there is a unique way to group the tiles in $P$ into supertiles. This 
implies that any return vector between two instances of $P$ must be a multiple 
of $|a|+|b|$, and hence that $\Ret(P)$ is an infinite cyclic group. 

Note that $\Ret(P)$ isn't necessarily all of $(|a|+|b|)\ZZ$. If $P$ is long enough to determine
the locations of the order-$n$ supertiles, then $\Ret(P)$ is actually contained in 
$2^{n-1}(|a|+|b|)\ZZ$. However, this does not affect the rank of $\Ret(P)$, which is 1 for all $P$ 
of length 5 or greater.

At first glance this result is surprising, since the rank of $\check H^1(\Omega) = \ZZ[1/2] \oplus \ZZ$ is 2, not 1. Theorem \ref{thm:main2} says that a generic
shape change should result in return modules that have rank 2. 

The answer is that part of the cohomology involves collared tiles and cannot 
be expressed in terms of the basic uncollared tiles. As we will soon see, a 
generic shape change among {\em collared} tiles does indeed result in 
return modules of rank 2 or higher for arbitrary patches, and of rank 2 for 
large patches. 

Since the patterns $aaa$ and $bbb$ never appear, 
there are six possible once-collared tiles: 
\begin{eqnarray*}
a_1 = (a)a(b), &\qquad & b_1 = (b)b(a), \cr 
a_2  =  (b)a(a), &\qquad & b_2  =  (a)b(b), \cr 
a_3  =  (b)a(b), &\qquad & b_3  =  (a)b(a),  
\end{eqnarray*}
where $(x)y(z)$ denotes a $y$ tile that is preceded by an $x$ tile and followed by 
a $z$ tile. 

The Anderson-Putnam complex that describes possible adjacencies between tiles 
is shown in Figure \ref{fig:thuemorse}. 

\begin{figure}[ht]
    \centering
    \includegraphics[width=0.2\textwidth]{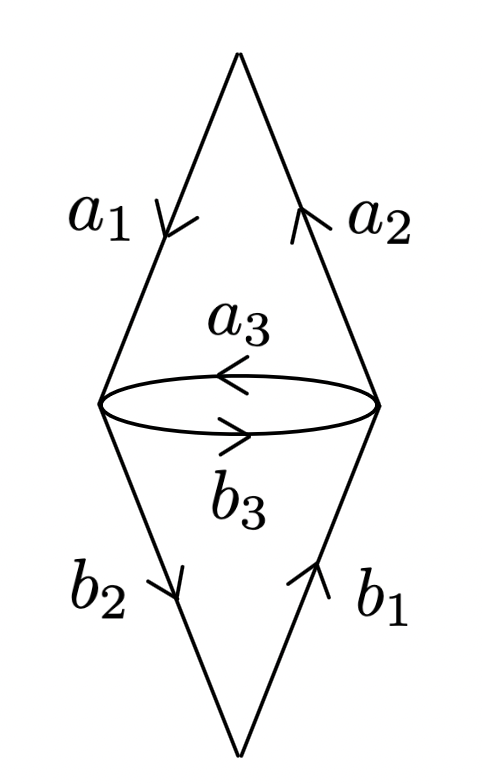}
    \caption{The Anderson-Putnam complex for the once-collared Thue-Morse substitution}
    \label{fig:thuemorse}
\end{figure}

A path between two instances of a tile 
defines a closed loop in this complex, so the length of an arbitrary return vector is 
an integer combination of 
\begin{eqnarray*} 
\gamma_1 &=& |a_1| + |a_2| - |a_3| \cr 
\gamma_2 &=& |b_1| + |b_2| - |b_3| \cr 
\gamma_3 &= & |a_3| + |b_3|. 
\end{eqnarray*} 
The quantities $\{\gamma_1, \gamma_2, \gamma_3\}$ are rationally 
independent given generic choices of tile lengths, 
so the return module for individual collared tiles has (generic) rank 3. 

However, return modules for larger patches involve closed loops in an analogous
complex built from supertiles, not just from individual tiles.
Substitution sends $\gamma_1$ and $\gamma_2$ to $\gamma_3$ and sends $\gamma_3$ to $\gamma_1+\gamma_2+\gamma_3$, so
the return vector for any patch containing a well-defined collared supertile
is a combination of $\gamma_1+\gamma_2$ and $\gamma_3$. The quantities $\gamma_1+\gamma_2$ and $\gamma_3$ 
transform under additional substitution via the matrix 
$\left ( \begin{smallmatrix} 0 & 1 \cr 2 & 1 \end{smallmatrix} \right )$. 
This is invertible over $\QQ$, so the rank of the return module for a high-order 
supertile is the same as the rank for an order-1 supertile. As a result, return 
modules of all large patches have maximal (and generic) rank 2. 
\end{example}

\begin{example}[Three-e Morse] Our final one-dimensional example is a generalization of Thue-Morse that we call Three-e-Morse. Consider the 
substitution $a \to aab$, $b \to bba$. This has features similar to both 
Thue-Morse and to the period-doubling substitution $a \to ab$, $b \to aa$, only
with a stretching factor of 3 instead of 2. 

Working with basic tiles, there are only two lengths to play with, so the
return module of a patch cannot have rank greater than 2. The substitution 
matrix $\left ( \begin{smallmatrix} 2 & 1 \cr 1 & 2 \end{smallmatrix} \right )$
is invertible over $\QQ$, so the rank of the return module of a high-order supertile (or of any patch found within a high-order supertile, in other words of 
any patch at all) has the same rank as the return module of a simple tile. 
This rank is 1 if $|a|/|b| \in \QQ$ and 2 if $|a|/|b| \not \in \QQ$. 

However, the rank of $\check H^1(\Omega) = \ZZ[1/3] \oplus \ZZ^2$ is 3, not 2, 
so a generic shape change should give us rank-3 return modules. As with the 
Thue-Morse tiling, part of the 
cohomology is invisible to uncollared tiles, so we need to use collared tiles 
to achieve this. There are now eight collared tiles: 
\begin{eqnarray*}
a_1 = (a)a(a), &\qquad& b_1 = (b)b(b), \cr 
a_2 = (a)a(b), &\qquad & b_2 = (b)b(a), \cr 
a_3  =  (b)a(a), &\qquad & b_3  =  (a)b(b), \cr 
a_4  =  (b)a(b), &\qquad & b_4  =  (a)b(a).  
\end{eqnarray*}
\begin{figure}[ht]
    \centering
    \includegraphics[width=0.2\textwidth]{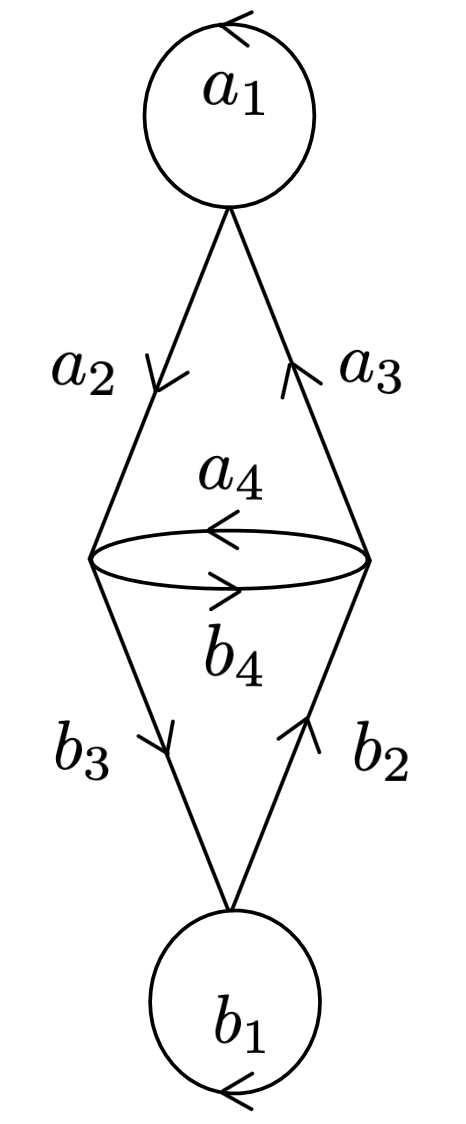}
    \caption{The Anderson-Putnam complex for the once-collared Three-e Morse tiling}
    \label{fig:threeemorse}
\end{figure}
The resulting Anderson-Putnam complex is shown in Figure \ref{fig:threeemorse}.

We define the quantities 
\begin{eqnarray*} 
\gamma_1 &=& |a_1|, \cr
\gamma_2 & = & |b_1|, \cr 
\gamma_3 &=& |a_2| + |a_3| - |a_4| \cr 
\gamma_4 &=& |b_3| + |b_3| - |b_4| \cr 
\gamma_5 & = & |a_4| + |b_4|. 
\end{eqnarray*} 
By picking generic tile lengths, we can get the $\gamma_i$'s to be rationally 
independent, so the rank of the return module of a single tile may have rank up to 5. 
Under substitution, the quantities $\gamma_1, \ldots, \gamma_5$ transform via the matrix 
\[ \begin{pmatrix} 0&0&0&0&1 \cr 0&0&0&0&1 \cr 1 &0&1&0&1 \cr 
0 & 1 & 0 & 1 & 1 \cr 1&1&1&1&1 \end{pmatrix}. \]
This matrix has rank 3, with eigenvalues 3, 1, -1, 0 and 0. Thus the return module
for order-1 supertiles (or for higher-order supertiles, or for any large patch)
will generically have rank 3. If the $\gamma_i$'s are not rationally independent, the
rank may be even less. 
\end{example} 

The lesson in all of these examples is that the rank of the cohomology gives 
an upper bound on the ranks of return modules for large patches and gives the
rank of such return modules exactly for a generic choice of tile lengths. If the 
cohomology can be expressed entirely using uncollared tiles, as in the Fibonacci
tiling and other Sturmian tilings, then we only need to vary the lengths of the 
uncollared tiles to achieve this. If part of the cohomology can only be expressed
using collared tiles, as in the Thue-Morse and Three-e-Morse tilings, then
we need to vary the lengths of collared tiles to achieve generic behavior. 
In all cases we can get return modules of small patches to be as complicated as 
we want by collaring and varying a large number of tile lengths. However, these
complications disappear when we look at large patches, typically by having a 
singular substitution matrix. 

\subsection{Higher dimensional examples} 

When it comes to return modules, there are several key 
differences between the geometry of $\RR$ and that of $\RR^d$ with
$d>1$. First, the return module of a repetitive tiling in $d$ 
dimensions always has rank at least $d$. Since every ball of 
sufficient radius contains a copy of the patch $P$ that we 
are studying, the real span of the return vectors of $P$ is all 
of $\RR^d$, so there must be $d$ return vectors that are linearly
independent over $\RR$, and thus over $\QQ$. 

Second, non-generic behavior is rarer in higher dimensions than
in one dimension, in the sense that it occurs on countably many copies of a codimension-$d$ set
instead of a codimension-1 set. (However, the word ``rarer'' should be taken with a grain of salt, since in both cases the sets have measure zero.) 
There are only countably many ways to 
have a rational linear relation among return vectors. Each such relation
reduces the number of free vectors by one, thereby constraining us to a 
codimension-$d$ subset of our set of possible shape parameters. 

Finally, the maximal rank of a return module is the rank of $\check H^1$, but
the dimension of the space of shape parameters is $d$ times the rank of 
$\check H^1$. There are many, many more ways to vary the shapes and sizes of 
our tiles than there are generators of our return modules.

\begin{example}[Chair] There are two common versions of the chair tiling, each MLD
to the other. In one, the basic tiles are an L-shaped triomino and rotations of 
that triomino by multiples of 90 degrees. The substitution rule for one tile is 
given in the top portion of Figure \ref{fig:chairarrow}. To substitute rotated versions of the tile, just 
rotate the picture. In the other version of the chair tiling, a basic tile is a 
unit square marked with an arrow pointing northeast, northwest, southwest, or 
southeast. The substitution rule for arrow tiles is shown in the bottom portion of Figure \ref{fig:chairarrow}. To go from chairs to arrows, divide each chair-shaped tile into three squares and draw arrows pointing towards the center. 
To go from arrows to chairs, look for vertices with three arrows pointing in and
one pointing out. Glue the three squares with arrows pointing in to form a chair.

\begin{figure}[ht]
    \centering
    \includegraphics[width=0.4\textwidth]{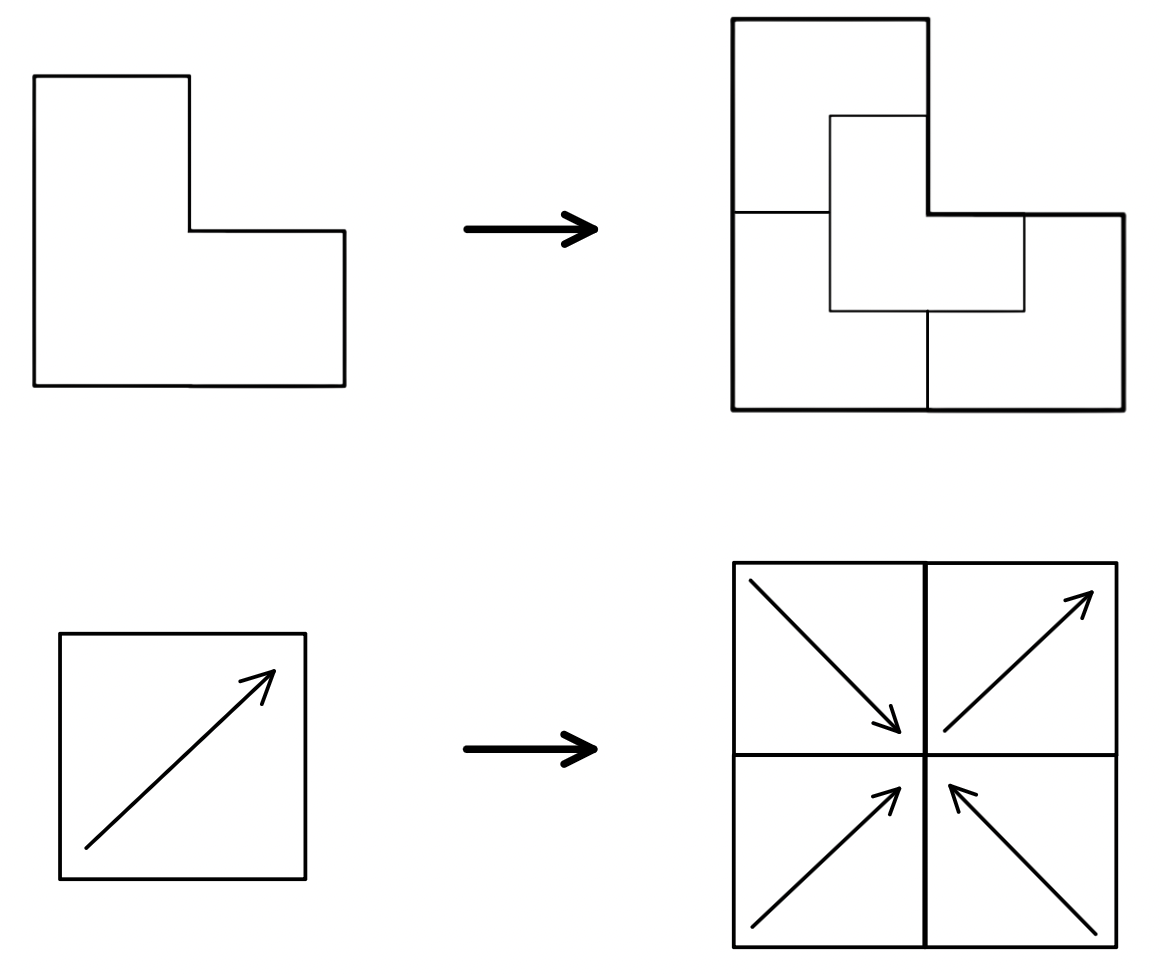}
    \caption{Substitutions for the classic and arrow versions of the chair tiling}
    \label{fig:chairarrow}
\end{figure}

In either version, all of the vertices have exactly the same $x$-coordinate (mod 1)
and the same $y$-coordinate (mod 1). The return module for any patch is necessarily
a subgroup of $\ZZ^2$ and so has rank at most 2. (Depending on 
the size of the patch, the return module is typically $2^n \ZZ^2$ for some 
integer $n$.) 
The chair tiling, like all primitive substitution tilings, is {\em repetitive}, 
meaning that for each patch $P$ there is a radius $r_P$ such that every ball of 
radius $r_P$ contains at least one copy of $P$. This implies that the rank of 
$\Ret(P)$ is at least equal to the dimension of the ambient space, in this case 2. 
Remarkably, it is no bigger. 

It is reasonable to wonder whether that rank can be increased by applying a shape
change. For small patches the answer is ``yes'', just as in our 1D examples. 
But for large patches the answer is ``no''. $\check H^1(\Omega) = \ZZ[1/2]^2$, 
which only has rank 2. By Theorem \ref{thm:main1}, no matter how we change the shapes, the return modules of 
large patches can only have rank 2. That is, the locations of an arbitrary large 
patch $P$ must live on a lattice, up to an overall translation. 

Another perspective on this comes from the classification of shape changes up
to MLD equivalence by $\check H^1(\Omega, \RR^2)$ \cite{CS1}. There is a 4-parameter
family of rigid linear transformations that can be applied to the chair tiling.
Since $\check H^1(\Omega, \RR^2) = \ZZ[1/2]^2 \otimes \RR^2 = \RR^4$ is 
4-dimensional, all shape changes are MLD to rigid linear transformations, and in
particular respect any lattice structure on any length scale larger than that 
of the MLD equivalence. 
\end{example}

\begin{example}[Hat] 
In 2023, Smith et al \cite{Hat} announced the discovery of a family of 
aperiodic monotiles. Each element in the family is a shape that, together 
with rotated and reflected versions of itself, can tile the plane but
only nonperiodically. Each tiling actually involves 12 tiles up 
to translation: the original shape rotated by multiples of 60 
degrees and the reflected shape rotated by multiples of 60 degrees. 
(The same authors later displayed an aperiodic monotile called 
the Spectre \cite{Spectre} that uses 12 rotations and no reflections. We will not
concern ourselves with the Spectre tilings here.) 

\begin{figure}[ht]
    \centering
    \includegraphics[width=0.4\textwidth]{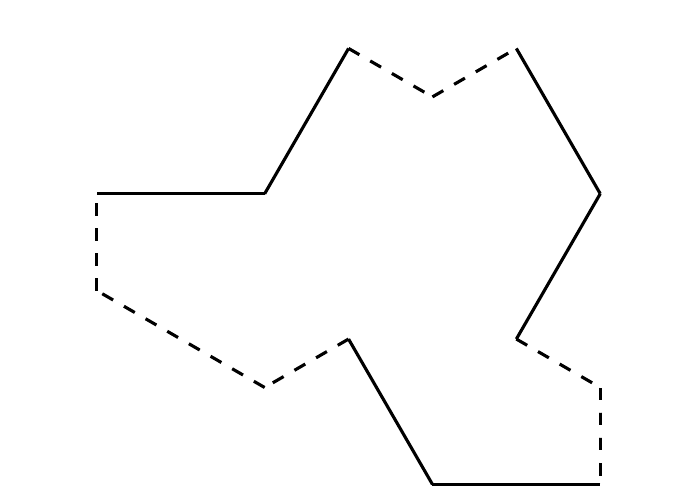}
    \caption{The basic ``Hat'' tile. The lengths $\alpha, \beta$
    of the solid and dashed edges are arbitrary.}
    \label{fig:Hat-shape}
\end{figure}

The basic tile, shown in Figure \ref{fig:Hat-shape}, 
is a degenerate 14-gon, with six edges of length 
$\alpha$ (shown as solid lines) and eight edges of length $\beta$ (shown as dashed lines), and with 
two of the dashed edges laid end-to-end, looking like a single edge of 
length $2\beta$.
The construction works for all positive values of 
$\alpha$ and $\beta$, although some care must be taken with the
special case $\alpha=\beta$ (to ensure that $\alpha$ edges cannot 
abut $\beta$ edges) and with the limiting cases $\alpha=0$ and 
$\beta=0$ (to ensure that the zero-length edges still line up).
Several special cases have been given names: $\beta=0$ 
is a ``Chevron'', $\alpha/\beta=\sqrt{3}$ is a ``Hat'', $\alpha=\beta$
is a ``Spectre'', $\alpha/\beta=\sqrt{3}/3$ is a ``Turtle'' and 
$\alpha=0$ is a ``Comet''. See Figure \ref{fig:Hat-family}

\begin{figure}[ht]
\begin{center}
\includegraphics[angle=180, origin=c, width=0.2\textwidth]{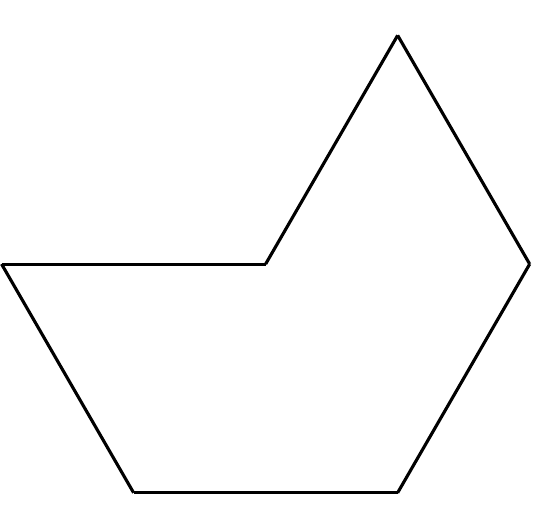} \qquad 
\includegraphics[angle=180, origin=c,width=0.3\textwidth]{HatBW.pdf}
\includegraphics[angle=180, origin=c,width=0.3\textwidth]{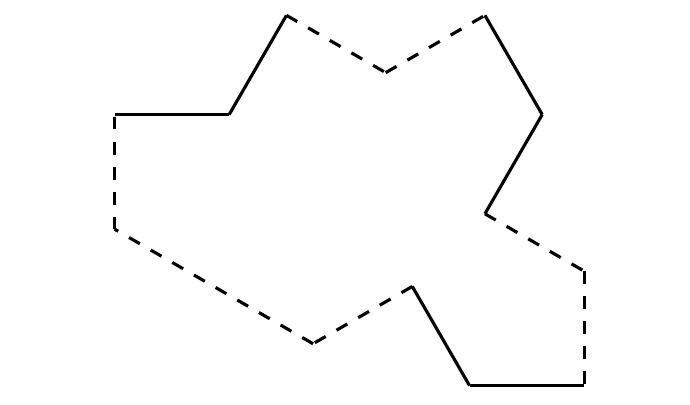} \\
\includegraphics[angle=180, origin=c,width=0.2\textwidth]{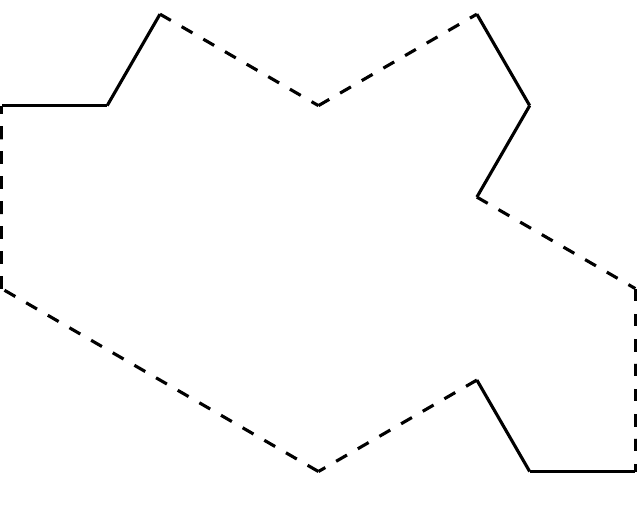}
\hskip 0.1 \textwidth
\includegraphics[angle=180, origin=c,width=0.2\textwidth]{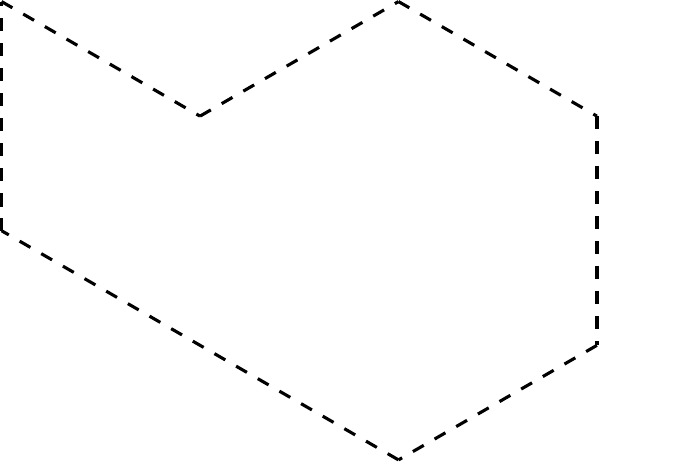}
\end{center}
\caption{The Chevron, Hat, Spectre, Turtle and Comet tiles, all rotated
by 180 degrees.}
\label{fig:Hat-family}
\end{figure}

In all cases, tiles must assemble into four shapes called ``metatiles''. The metatiles then assemble into larger metatiles,
which assemble into larger metatiles, and so on to infinity. 
The substitution involving metatiles has an overall stretching
factor of $\phi^4$, where $\phi$ is the golden mean. Since 
$\phi^4$ is a {\em unimodular} Pisot number, this implies that 
the return module for high-order metatiles (and other large patches) 
is the same as for basic metatiles, just like what we saw with
the Fibonacci tiling. We refer to this as the return module of 
the tiling. While it is possible to investigate the (possibly larger)
return module for individual tiles, we will not do so. 

The cohomology and return modules for tilings in the Hat family
were computed in \cite{BGS}. $\check H^1(\Omega)=\ZZ^4$, so the biggest rank the 
return module can have is 4. Our ambient space is 2-dimensional, so the return
module must have rank at least 2. 6-fold rotational symmetry implies that 
the rank of the return module must be even. In short, the rank must always be 
2 or 4. 

If we identify $\RR^2$ 
with the complex numbers $\CC$, then the return module is closely related to 
the triangular lattice $\ZZ[\xi]$ spanned by 1 and $\xi = \exp(2\pi i/6)
= (1 + i \sqrt{3})/2$. Specifically, the return lattice is the span of 
$(\alpha+i\beta) (1+\xi) \ZZ[\xi]$ and $2i\beta (1+\xi) \ZZ[\xi]$. This has 
rank 2 if a rational multiple of $2i\beta/\alpha$ lies in $\ZZ[\xi]$ and rank
4 if it does not. Since the pure imaginary elements of $\ZZ[\xi]$ are the 
multiples of $i \sqrt{3}$, the return module has rank 2 if and only if 
$\beta \sqrt{3}/\alpha$ is rational or $\alpha=0$. 

The upshot is that almost every shape in the Hat family results in a 4-dimensional
return module. However, four of the five named variants have 2-dimensional 
return modules. In two of these, namely the Chevron and the Comet, all vertices 
lie on a triangular lattice, so the return module is necessarily contained in 
that lattice. In two others, namely the Hat and the Turtle, all vertices either 
lie on a triangular lattice or on its dual honeycomb grid.

So far we have only considered the shapes discussed in \cite{Hat}, all of 
which respect 6-fold rotational symmetry and reflectional symmetry. In \cite{BGS}, 
the authors considered shape changes that broke reflectional symmetry while
preserving rotational symmetry, resulting in two distinct shapes, each appearing
in 6 orientations. This was done by allowing the parameters $\alpha$ and $\beta$ 
to be complex. As before, the return module has rank 2 exactly when 
$i \beta/\alpha \in \QQ[\xi]$, the set of rational multiples of $\ZZ[\xi]$. 
The set $\QQ[\xi]$ is countable and dense in $\CC$. Our space of shape 
parameters $(\alpha, \beta)$ is $\CC^2 =\RR^4$, most of which give rise to 
rank-4 return modules. Only when $(\alpha, \beta)$ lie in a countable union of 
2-dimensional subsets do we get rank 2. 

Expanding our horizons further, we can consider shape changes that break
6-fold rotational symmetry. The space of possible shapes (up to MLD equivalence) 
is parametrized by $\check H^1(\Omega,\RR^2) = \ZZ^4 \otimes \RR^2 = \RR^8$. 
Within this 8-dimensional space, the return module has rank 4 everywhere except on a 
countable union of 6-dimensional subsets, where it can have rank 2 or 3. 
\end{example}

\section*{Acknowledgements}

It is our pleasure to thank Michael Baake, Franz G\"ahler and 
Jianlong Liu for 
useful discussions. This work was supported by the National Science Foundation under grant DMS-2113468. 

\end{document}